\documentclass[12pt,a4paper]{article}

\usepackage[margin=2.7cm]{geometry}
\usepackage{PRIMEarxiv}
\usepackage{amsmath,amssymb,amsthm}
\usepackage{mathtools}
\usepackage{array}
\usepackage[hidelinks]{hyperref}
\usepackage{orcidlink}
\usepackage{tikz}
\usepackage{pgfplots}
\pgfplotsset{compat=1.18}
\usepgfplotslibrary{groupplots}
\newcommand{\plotbeta}{1}

\newtheorem{theorem}{Theorem}
\newtheorem{proposition}{Proposition}
\newtheorem{corollary}{Corollary}
\newtheorem{remark}{Remark}
\newtheorem{definition}{Definition}

\newcolumntype{C}[1]{>{\centering\arraybackslash}p{#1}}

\title{Exact Uniform L1 Spacing for Solow--Polasky Diversity on Lines and Ordered Pareto Fronts}

\author{%
\normalfont
\begin{tabular}{@{}C{0.30\textwidth}C{0.30\textwidth}C{0.30\textwidth}@{}}
\textbf{Michael T. M. Emmerich~\orcidlink{0000-0002-7342-2090}} &
\textbf{Mahboubeh Nezhadmoghaddam~\orcidlink{0009-0001-7637-7635}} &
\textbf{Jes\'us Guillermo Falc\'on Cardona~\orcidlink{0000-0003-1131-098X}}\\[0.6ex]
{\small Faculty of Information Technology} &
{\small School of Engineering and Sciences} &
{\small Computer Science Department}\\
{\small University of Jyv\"askyl\"a} &
{\small Tecnol\'ogico de Monterrey} &
{\small CICESE Research Center}\\
{\small Jyv\"askyl\"a, Finland} &
{\small Monterrey, M\'exico} &
{\small Ensenada, Baja California, M\'exico}\\[0.4ex]
{\small \href{mailto:michael.t.m.emmerich@jyu.fi}{michael.t.m.emmerich@jyu.fi}} & &
\end{tabular}%
}

\date{\today}

\begin{document}
\maketitle

\begin{abstract}
We study fixed-cardinality maximization of the inverse-matrix Solow--Polasky diversity, equivalently finite metric magnitude for the exponential kernel, on one-dimensional and ordered metric sets.  The analysis starts from the known finite-line gap formula for the exponential kernel, which writes the excess inverse-matrix diversity as a sum of functions of consecutive gaps.  Building on this formula, the main interval theorem proves that, for every $k\geq 2$, the unique maximizing $k$-point subset of $[0,1]$ is the equally spaced set.  Thus the objective selects a uniform gap representation on the real line.  A converse kernel proposition shows that, among normalized non-increasing distance kernels, requiring the corresponding adjacent-gap additive structure forces the exponential family.  Further results transfer the interval theorem to ordered $\ell_1$ (L1, or Manhattan) curves by isometry: the maximizing sets are uniform in accumulated $\ell_1$ length.  As a consequence, monotone biobjective Pareto fronts admit Solow--Polasky optimal finite approximations that are uniformly spaced in accumulated objective-space change, a natural representation when all parts of a continuous front should be covered.  Examples, including a dense connected front and a finite disconnected ZDT3 front, illustrate how the continuous uniform-gap result appears on discrete candidate sets.
\end{abstract}


\section{Introduction}
\label{sec:introduction}

Let $\beta>0$ be fixed.  For a finite subset $S=\{x_1,\ldots,x_k\}$ of a metric space $(X,d)$, define the exponential similarity matrix
\[
 Z_S = \bigl(e^{-\beta d(x_i,x_j)}\bigr)_{i,j=1}^k .
\]
Whenever $Z_S$ is nonsingular, the Solow--Polasky diversity of $S$ at scale $\beta$ is
\[
 D_\beta(S) = \mathbf{1}^\top Z_S^{-1}\mathbf{1}.
\]
This inverse-matrix quantity was introduced in biodiversity conservation by Solow, Polasky, and Broadus~\cite{SolowPolaskyBroadus1993}.  It has since also appeared as a diversity criterion in discrete diversity optimization~\cite{UlrichBaderThiele2010,PereverdievaEtAl2025} and, more recently, in multiobjective optimization contexts~\cite{Huntsman2023}.  In metric geometry the same expression is the magnitude of the finite metric space after scaling the distance by $\beta$; in the finite exponential-kernel setting, Solow--Polasky diversity and magnitude are therefore mathematically identical, although they arise from different motivations~\cite{LeinsterMeckes2016}.

We consider the fixed-cardinality maximization problem: for a metric set $X$ and an integer $k$, choose a $k$-point subset $S\subset X$ maximizing $D_\beta(S)$.  In a general metric space this is a combinatorial optimization problem.  The real line, however, has an additional ordered structure.  If
\[
0\leq x_1<x_2<\cdots<x_k\leq 1
\]
and $\delta_i=x_{i+1}-x_i$, then every distance $d(x_i,x_j)$ is the sum of the consecutive gaps between $x_i$ and $x_j$.  For the exponential kernel this turns similarities into products over adjacent gaps.  The resulting adjacent-gap formula for the inverse-matrix objective is known: in magnitude theory it is the finite-line formula of Leinster and Willerton~\cite{LeinsterWillerton2009}, and the fixed-scale gap-sum form needed for Solow--Polasky diversity is covered by~\cite{Emmerich2026PolynomialTime}.  We take this formula as the starting point of the present analysis.  The question addressed here is what the known gap decomposition implies for fixed-cardinality optimization: does maximizing the inverse-matrix objective force an optimal finite set to distribute its gaps uniformly?

The main result answers this question exactly.  On $[0,1]$, the fixed-cardinality maximizer of $D_\beta$ is the equally spaced set, and for $k\geq 2$ this maximizing $k$-point subset is unique.  Equivalently, the inverse-matrix objective selects the uniform gap representation of the interval.  A converse result complements the theorem: among normalized non-increasing distance kernels, requiring the exact adjacent-gap additive structure of the baseline-corrected inverse-matrix diversity forces the exponential kernel family.

A further consequence is obtained by isometry.  Throughout, $\ell_1$ denotes the L1, or Manhattan, distance.  Ordered $\ell_1$ curves have additive accumulated length, and hence can be parameterized by an interval without changing distances along the curve.  The interval theorem therefore transfers to such curves: optimal $k$-point subsets are uniformly spaced in accumulated $\ell_1$ length.  As an application, monotone biobjective Pareto fronts can be viewed in this way under the $\ell_1$ objective-space metric.  Finite approximations of continuous Pareto fronts are often sought to represent all parts of the front; in this setting, Solow--Polasky diversity gives an exact mathematical justification for uniform spacing in accumulated objective-space change.  This observation may be useful, for example, when finite populations or archives are used as discrete Pareto-front approximations in multiobjective optimization.

The remainder is organized as follows.  Section~\ref{sec:line-theory} contains the line formula, the uniform-spacing theorem and proof, and the converse kernel result.  Section~\ref{sec:ordered-l1-curves} gives the ordered-$\ell_1$ reduction and the Pareto-front interpretation.  Section~\ref{sec:finite-front-examples} gives two finite-front examples.  The appendices keep the Jensen derivation and the expanded derivation of the converse kernel result.

\section{Line theory and kernel characterization}
\label{sec:line-theory}

We first work in the metric space where the proof mechanism is explicit: a real interval with its usual distance.  This keeps the additive gap structure visible before it is transferred to ordered $\ell_1$ curves.

\subsection{Uniform spacing on an interval}
\label{sec:unit-line-theory}

We use the following line formula.  In magnitude theory it is the finite-line formula of Leinster and Willerton~\cite{LeinsterWillerton2009}; in the Solow--Polasky setting, the fixed-scale gap-sum form used here is also proved in~\cite{Emmerich2026PolynomialTime}.  If
\[
 0 \leq x_1 < x_2 < \cdots < x_k \leq 1
\]
and $\delta_i=x_{i+1}-x_i$, then
\begin{equation}
\label{eq:line-formula}
D_\beta(\{x_1,\ldots,x_k\})
 = 1 + \sum_{i=1}^{k-1} \tanh\!\left(\frac{\beta\delta_i}{2}\right).
\end{equation}

We briefly recall why the inverse-matrix Solow--Polasky objective reduces to this gap sum on the real line.  For $i<j$, the distance from $x_i$ to $x_j$ is the sum of the consecutive gaps between them,
\[
 d(x_i,x_j)=\sum_{r=i}^{j-1}\delta_r .
\]
Consequently, if $a_i=e^{-\beta\delta_i}$, then the exponential similarity kernel factorizes as
\[
 e^{-\beta d(x_i,x_j)}
 = \prod_{r=i}^{j-1} a_r .
\]
This product structure is special to ordered points on a line.  It yields a tridiagonal inverse similarity matrix, and summing the corresponding row sums gives
\[
D_\beta(\{x_1,\ldots,x_k\})
=
1+\sum_{i=1}^{k-1}\frac{1-a_i}{1+a_i}.
\]
Each gap contribution can be written as
\[
\frac{1-a_i}{1+a_i}
=
\frac{a_i^{-1/2}(1-a_i)}{a_i^{-1/2}(1+a_i)}
=
\frac{a_i^{-1/2}-a_i^{1/2}}
 {a_i^{-1/2}+a_i^{1/2}}
=
\tanh\!\left(\frac{\beta\delta_i}{2}\right).
\]
Thus, on the real line, the inverse-matrix objective separates into independent contributions of consecutive gaps.  In particular, after the points are ordered, the fixed-cardinality optimization problem becomes a gap-allocation problem.  This one-dimensional structure is also the basis of the scaled Solow--Polasky diversity dynamic program for ordered finite $\ell_1$ instances~\cite{Emmerich2026PolynomialTime} and is consistent with the $\ell_1$ magnitude viewpoint of Leinster and Meckes~\cite{LeinsterMeckes2017}.

\begin{theorem}[Uniform spacing on the unit interval]
\label{thm:unit-line-uniform}
Let $k\geq 2$ and $\beta>0$.  Among all $k$-point subsets of the unit interval $[0,1]$ with the usual distance, the Solow--Polasky diversity $D_\beta$ is maximized by the uniformly spaced set
\[
 S_k^\star = \left\{0,\frac{1}{k-1},\frac{2}{k-1},\ldots,1\right\}.
\]
The maximum value is
\[
 D_\beta(S_k^\star)
 = 1 + (k-1)\tanh\!\left(\frac{\beta}{2(k-1)}\right).
\]
Moreover, for $k\geq 2$, the maximizing $k$-point subset is unique and is given by $S_k^\star$.
\end{theorem}

\begin{proof}
Let $S=\{x_1<\cdots<x_k\}\subseteq[0,1]$ and let $\delta_i=x_{i+1}-x_i$ for $i=1,\ldots,k-1$.  By the line formula \eqref{eq:line-formula}, maximizing $D_\beta(S)$ is equivalent to maximizing
\[
 \sum_{i=1}^{k-1} f(\delta_i),
 \qquad
 f(t)=\tanh\!\left(\frac{\beta t}{2}\right),
\]
subject to $\delta_i>0$ and $\sum_{i=1}^{k-1}\delta_i\leq 1$.

The function $f$ is strictly increasing on $[0,\infty)$.  Hence an optimal set must have total span one: otherwise, its gaps could be increased so that the selected points span the full unit interval, and the objective would increase.  Thus, at optimality,
\[
 \sum_{i=1}^{k-1}\delta_i=1.
\]
The function $f$ is also strictly concave on $(0,\infty)$, since
\[
 f''(t)
 = -\frac{\beta^2}{2}\,\operatorname{sech}^2\!\left(\frac{\beta t}{2}\right)
      \tanh\!\left(\frac{\beta t}{2}\right) <0
 \qquad (t>0).
\]
Figure~\ref{fig:f-and-second-derivative} visualizes this increasing and strictly concave behavior by plotting both $f$ and $f''$ for the representative scale value $\beta=\plotbeta$.
Therefore Jensen's inequality (Appendix~\ref{app:jensen}; see also~\cite{HardyLittlewoodPolya1952}) gives
\[
 \frac{1}{k-1}\sum_{i=1}^{k-1} f(\delta_i)
 \leq
 f\!\left(\frac{1}{k-1}\sum_{i=1}^{k-1}\delta_i\right)
 = f\!\left(\frac{1}{k-1}\right).
\]
Multiplying by $k-1$ yields
\[
 \sum_{i=1}^{k-1} f(\delta_i)
 \leq
 (k-1)f\!\left(\frac{1}{k-1}\right).
\]
In other words, the uniformly spaced gap vector gives an objective value no smaller than any other feasible gap vector.  Because $f$ is strictly concave, this inequality is strict unless all gaps are equal.  Equality in Jensen's inequality holds if and only if
\[
 \delta_1=\delta_2=\cdots=\delta_{k-1}=\frac{1}{k-1}.
\]
Together with the full-span condition this gives $x_1=0$ and $x_k=1$, and hence the uniformly spaced set $S_k^\star$.
\end{proof}

\begin{remark}[Direct evaluation on ordered sets]
Equation~\eqref{eq:line-formula} is also computationally useful.  Once the selected points are given in their line order, the Solow--Polasky diversity does not have to be evaluated by forming and inverting the full similarity matrix.  It is enough to compute the consecutive gaps and sum the $k-1$ terms in~\eqref{eq:line-formula}.  Thus, for an already ordered $k$-point set on a line, or on an ordered $\ell_1$ curve after applying the scalar coordinate used below, the value can be computed in linear time in $k$.
\end{remark}

\begin{remark}[Relation to minimum-gap criteria]
Maximizing the minimum pairwise distance on a line is equivalent to maximizing the smallest consecutive gap.  This bottleneck objective also leads to uniform gaps on the continuous interval, and it may be viewed as the limiting criterion obtained from Riesz-type repulsion when the exponent tends to infinity.  However, it only reacts to the closest pair.  Once the minimum gap is fixed, improvements in other gaps or in more distant pairs are invisible to the objective.  One can refine the criterion by a lexicographic tie-breaker on the ordered gap vector, or by adding a small secondary term, but such augmentations are less direct and can be numerically delicate.  The Solow--Polasky objective avoids this pure bottleneck behavior: through the concave tanh-sum in~\eqref{eq:line-formula}, every adjacent gap contributes to the diversity value.
\end{remark}

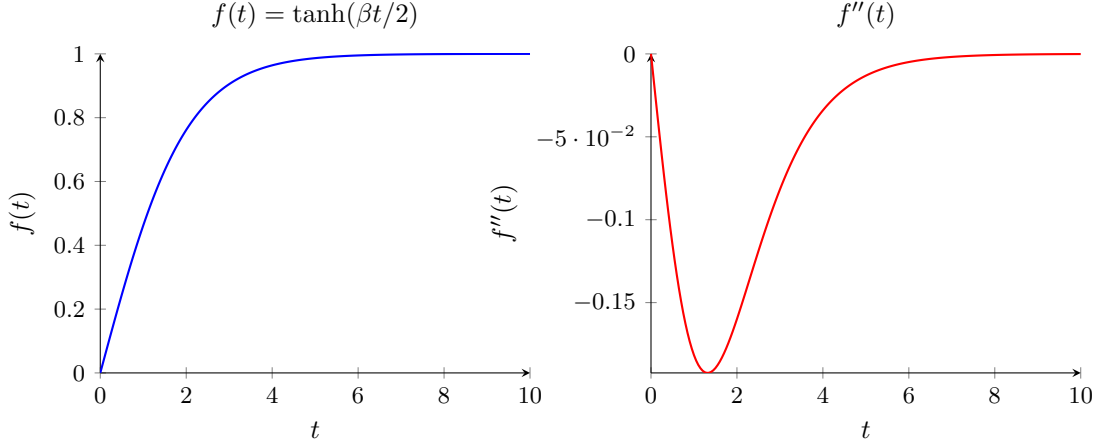
\begin{figure}[t]
\centering
\begin{tikzpicture}
\begin{groupplot}[
    group style={group size=2 by 1, horizontal sep=1.6cm},
    width=0.44\textwidth,
    height=5.8cm,
    domain=0:10,
    samples=201,
    axis lines=left,
    xmin=0, xmax=10,
    enlargelimits=false,
    tick label style={font=\small},
    label style={font=\normalsize},
    every axis title/.append style={font=\normalsize},
]

\nextgroupplot[
    xlabel={$t$},
    ylabel={$f(t)$},
    title={$f(t)=\tanh(\beta t/2)$},
]
\addplot[thick, blue]
{
    (exp(\plotbeta*x/2)-exp(-\plotbeta*x/2)) /
    (exp(\plotbeta*x/2)+exp(-\plotbeta*x/2))
};

\nextgroupplot[
    xlabel={$t$},
    ylabel={$f''(t)$},
    title={$f''(t)$},
]
\addplot[thick, red]
{
    -(\plotbeta^2/2) *
    (4/((exp(\plotbeta*x/2)+exp(-\plotbeta*x/2))^2)) *
    ((exp(\plotbeta*x/2)-exp(-\plotbeta*x/2)) /
     (exp(\plotbeta*x/2)+exp(-\plotbeta*x/2)))
};

\end{groupplot}
\end{tikzpicture}
\caption{The gap contribution function $f(t)=\tanh(\beta t/2)$ and its second derivative on the domain $0\leq t\leq 10$, shown for $\beta=\plotbeta$.  The positivity of $f$ and negativity of $f''$ away from $t=0$ visualize the increasing, strictly concave behavior used in the Jensen argument.}
\label{fig:f-and-second-derivative}
\end{figure}

The interval theorem uses the exponential kernel in an essential way.  Before passing to ordered $\ell_1$ curves, we record the converse result showing that exact adjacent-gap additivity forces this kernel family.

\subsection{Uniqueness of the exponential kernel}
\label{sec:exponential-kernel-uniqueness}

The uniform-spacing result above relies on the finite-line formula~\eqref{eq:line-formula}.  For the exponential kernel, this formula says that the diversity in excess of the singleton baseline is additive in adjacent gaps.  The baseline matters: for every normalized kernel with $K(0)=1$, a singleton has similarity matrix $[1]$ and hence
\[
  D_K(\{x\})=\mathbf{1}^\top [1]^{-1}\mathbf{1}=1.
\]
Thus the natural additive quantity is
\[
  E_K(S):=D_K(S)-1.
\]
The correction by $1$ is forced.  If a three-point gap identity were written with an arbitrary constant $c$,
\[
  D_K(\{0,a,a+b\})
  =D_K(\{0,a\})+D_K(\{0,b\})-c,
\]
then letting $b\downarrow0$ gives $D_K(\{0,b\})\to D_K(\{0\})=1$ and the left-hand side reduces to $D_K(\{0,a\})$.  Hence $c=1$.  A zero-length gap therefore contributes no excess diversity only under the singleton-baseline normalization.

The following converse shows that this additive mechanism is specific to the exponential family.  It does not select one universal numerical scale; rather, once the kernel is fixed, the value $K(1)$ determines the parameter.  Appendix~\ref{app:converse-details} gives a more detailed derivation, including the continuous Cauchy-equation step.

\begin{proposition}[Adjacent-gap additivity forces the exponential family]
\label{prop:converse-exponential-kernel}
Let $K:[0,\infty)\to(0,1]$ be continuous, non-increasing, normalized by $K(0)=1$, and satisfy $K(t)<1$ for every $t>0$.  For a finite ordered set $S=\{x_1<\cdots<x_k\}\subset\mathbb{R}$, define
\[
  (Z_K(S))_{ij}=K(|x_i-x_j|),
  \qquad
  D_K(S)=\mathbf{1}^\top Z_K(S)^{-1}\mathbf{1},
\]
whenever $Z_K(S)$ is nonsingular, and put $E_K(S)=D_K(S)-1$.  Suppose that this excess diversity is additive in adjacent gaps for every three-point set:
\begin{equation}
\label{eq:three-point-excess-additivity}
  E_K(\{0,a,a+b\})
  =E_K(\{0,a\})+E_K(\{0,b\})
  \qquad (a,b>0).
\end{equation}
Then $K$ belongs to the exponential family.  More precisely,
\[
  K(t)=e^{-\beta t}\qquad(t\geq0),
\]
where the parameter is uniquely fixed by the one-step value,
\[
  \beta=-\log K(1)>0.
\]
\end{proposition}

\begin{proof}
For two points at distance $t$, writing $r=K(t)$ gives
\[
  D_K(\{0,t\})
  =\mathbf{1}^\top
  \begin{pmatrix}1&r\\ r&1\end{pmatrix}^{-1}
  \mathbf{1}
  =\frac{2}{1+r},
\]
and hence
\[
  E_K(\{0,t\})=\frac{1-r}{1+r}.
\]
Fix $a,b>0$ and write
\[
  u=K(a),\qquad v=K(b),\qquad w=K(a+b).
\]
For the three-point set $\{0,a,a+b\}$, the similarity matrix is
\[
  Z_K=
  \begin{pmatrix}
  1 & u & w\\
  u & 1 & v\\
  w & v & 1
  \end{pmatrix}.
\]
Substituting this matrix and the two-point formula into~\eqref{eq:three-point-excess-additivity}, and simplifying, gives
\[
  (uv-w)\bigl(u^2-uvw-uv+v^2+w-1\bigr)=0.
\]
Thus either $w=uv$, or
\[
  w=\frac{1+uv-u^2-v^2}{1-uv}.
\]
The second alternative is impossible under the non-increasing assumption.  Interchanging $a$ and $b$ if necessary, assume $u\geq v$.  Since $a,b>0$, the assumptions give $0<v\leq u<1$.  Since $a+b>b$ and $K$ is non-increasing, we also have $w\leq v$.  However, the second alternative gives
\[
  w-v
  =\frac{(1-u)(u+1-v-v^2)}{1-uv}>0,
\]
because $u\geq v$ implies $u+1-v-v^2\geq1-v^2>0$.  This contradicts $w\leq v$.  Therefore only the first alternative remains, and
\[
  K(a+b)=K(a)K(b)
  \qquad(a,b>0).
\]
Together with $K(0)=1$, this identity holds on $[0,\infty)$.

Since $K$ is positive, $h(t)=\log K(t)$ is continuous and satisfies
\[
  h(a+b)=h(a)+h(b).
\]
The continuous Cauchy equation on $[0,\infty)$ gives $h(t)=t h(1)$.  Hence
\[
  K(t)=e^{t\log K(1)}=e^{-\beta t},
  \qquad
  \beta=-\log K(1).
\]
Because $K(1)<1$, we have $\beta>0$, and the displayed formula shows that this value of $\beta$ is unique for the given kernel.
\end{proof}

\section{Ordered \texorpdfstring{$\ell_1$}{L1} curves and Pareto fronts}
\label{sec:ordered-l1-curves}

The preceding section is purely one-dimensional.  We now show that ordered $\ell_1$ curves are one-dimensional metric objects in disguise, so the interval theorem applies without changing the Solow--Polasky objective.

\begin{definition}[Ordered $\ell_1$ curve]
A continuous curve $\gamma:[a,b]\to\mathbb{R}^m$ is called ordered in the $\ell_1$ sense if each coordinate function $\gamma_r$ is monotone on $[a,b]$.  For each coordinate choose $\sigma_r\in\{-1,+1\}$ such that $\sigma_r\gamma_r$ is nondecreasing.
\end{definition}

For such a curve define the scalar coordinate
\[
 \phi(t)=\sum_{r=1}^m \sigma_r\bigl(\gamma_r(t)-\gamma_r(a)\bigr).
\]
Then $\phi$ is nondecreasing, and its endpoint value
\[
 L=\phi(b)=\sum_{r=1}^m |\gamma_r(b)-\gamma_r(a)|
\]
is the $\ell_1$ distance between the two endpoints.  If the curve has no constant pieces in the $\ell_1$ metric, then $\phi$ is strictly increasing; otherwise the same statement applies after identifying points at zero $\ell_1$ distance along constant pieces.

\begin{proposition}[Ordered $\ell_1$ curves are intervals]
\label{prop:l1-isometry}
Let $\gamma:[a,b]\to\mathbb{R}^m$ be an ordered $\ell_1$ curve.  For all $s,t\in[a,b]$,
\[
 \|\gamma(t)-\gamma(s)\|_1 = |\phi(t)-\phi(s)|.
\]
Consequently, the image of $\gamma$ with the $\ell_1$ metric is isometric to an interval of length $L$ in the real line.
\end{proposition}

\begin{proof}
Assume without loss of generality that $s<t$.  Since each $\sigma_r\gamma_r$ is nondecreasing,
\[
 |\gamma_r(t)-\gamma_r(s)|
 = \sigma_r\bigl(\gamma_r(t)-\gamma_r(s)\bigr)
\]
for every coordinate $r$.  Summing over all coordinates gives
\[
 \|\gamma(t)-\gamma(s)\|_1
 = \sum_{r=1}^m |\gamma_r(t)-\gamma_r(s)|
 = \sum_{r=1}^m \sigma_r\bigl(\gamma_r(t)-\gamma_r(s)\bigr)
 = \phi(t)-\phi(s).
\]
This proves the distance identity.  The map $\gamma(t)\mapsto\phi(t)$ therefore preserves all pairwise $\ell_1$ distances, so the curve image is isometric to the interval $[0,L]$.
\end{proof}

\begin{theorem}[Uniform $\ell_1$ spacing on ordered curves]
\label{thm:l1-curve-uniform}
Let $\gamma:[a,b]\to\mathbb{R}^m$ be a continuous ordered $\ell_1$ curve of total $\ell_1$ length $L>0$.  For $k\geq2$, the $k$-point subset of the curve image maximizing Solow--Polasky diversity under the $\ell_1$ norm is obtained by choosing points $y_i=\gamma(t_i)$ such that
\[
 \phi(t_i)=\frac{(i-1)L}{k-1},
 \qquad i=1,
\ldots,k.
\]
Equivalently, consecutive selected points are uniformly spaced in the $\ell_1$ metric:
\[
 \|y_{i+1}-y_i\|_1=\frac{L}{k-1},
 \qquad i=1,\ldots,k-1.
\]
The maximum diversity value is
\[
 1+(k-1)\tanh\!\left(\frac{\beta L}{2(k-1)}\right).
\]
\end{theorem}

\begin{proof}
By Proposition~\ref{prop:l1-isometry}, the ordered $\ell_1$ curve is isometric to the interval $[0,L]$ through the coordinate $\phi$.  Solow--Polasky diversity depends only on the pairwise metric distances, so maximizing diversity over $k$ points on the curve is equivalent to maximizing it over $k$ points on $[0,L]$.  Scaling Theorem~\ref{thm:unit-line-uniform} from the unit interval to an interval of length $L$ gives the stated points and the stated maximum value.
\end{proof}

\begin{corollary}[Biobjective Pareto fronts]
\label{cor:biobjective-pareto-fronts}
Let $\gamma(t)=(f_1(t),f_2(t))$ be a continuous biobjective Pareto front parameterization with $f_1$ nondecreasing and $f_2$ nonincreasing.  Then, under the $\ell_1$ norm, the front is isometric to a line interval through
\[
 \phi(t)=\bigl(f_1(t)-f_1(a)\bigr)+\bigl(f_2(a)-f_2(t)\bigr).
\]
Therefore the fixed-cardinality Solow--Polasky diversity optimum on the continuous front consists of points that are uniformly spaced in $\ell_1$ distance along the front.
\end{corollary}

\begin{remark}[Connection with diversity optimization]
The corollary gives a continuous-front analogue of diversity optimization in multiobjective search.  It is related to indicator-based Solow--Polasky diversity maximization in diversity optimization and multiobjective optimization~\cite{UlrichBaderThiele2010}, to magnitude-based diversity enhancement~\cite{Huntsman2023}, and to results on optimum distribution properties of diversity indicators~\cite{PereverdievaEtAl2025,Nezhadmoghaddam2026riesz}.
\end{remark}

\begin{remark}[Relation to earlier observations on uniformity]
The conclusion should be read as a precise fixed-cardinality result for an ordered one-dimensional metric structure, not as the first observation that Solow--Polasky diversity or magnitude promotes well-spread point sets.  Ulrich, Bader, and Thiele observed in multiobjective search that the Solow--Polasky indicator can favor grid-like configurations~\cite{UlrichBaderThiele2010}, and Huntsman used magnitude-based ideas for diversity enhancement in evolutionary multiobjective optimization~\cite{Huntsman2023}.  More broadly, the magnitude viewpoint also makes clear that diversity and uniformity are related but not identical notions~\cite{LeinsterMeckes2016,Huntsman2023}.  The contribution of the present note is the exact fixed-cardinality statement for the exponential kernel on a line: the finite-line magnitude formula~\cite{LeinsterWillerton2009} reduces the objective to a sum of strictly concave gap terms, hence the continuous optimum has equal consecutive gaps; ordered $\ell_1$ Pareto fronts inherit this statement by an isometry to an interval.
\end{remark}

\begin{remark}[Finite candidate sets]
The preceding theorem concerns selection from a continuous ordered curve.  If selection is restricted to a prescribed finite candidate set on such a curve, exact uniform spacing may be impossible because the required scalar coordinates $iL/(k-1)$ need not be present.  In that discrete case, the isometry still reduces the problem to subset selection on an ordered finite set on the line.  An exact polynomial-time dynamic programming algorithm for this ordered finite $\ell_1$ setting, including biobjective Pareto fronts, is developed in~\cite{Emmerich2026PolynomialTime}.  This positive ordered-case result should be distinguished from the general fixed-cardinality problem, which is NP-hard in general metric spaces and already NP-hard in the Euclidean plane~\cite{EmmerichPereverdievaDeutz2026Metric,EmmerichPereverdievaDeutz2026Plane}.
\end{remark}

\section{Finite-front examples}
\label{sec:finite-front-examples}

The two examples below illustrate how the continuous equal-gap principle appears when selection is restricted to a finite candidate set.  In both cases the ordered $\ell_1$ coordinate reduces the computation to an ordered finite line instance, and the scaled gap-sum dynamic program of~\cite{Emmerich2026PolynomialTime} is used to solve the resulting fixed-cardinality subset-selection problem exactly.

\subsection{A dense connected finite-front example}

The continuous theory predicts exact uniform spacing in the induced line coordinate $\phi$.  For a finite but dense candidate set, exact uniform spacing may not be available, and the polynomial-time scaled Solow--Polasky gap-sum dynamic program from~\cite{Emmerich2026PolynomialTime} solves the fixed-cardinality subset-selection problem to find the best feasible approximation to that pattern.

Consider the $70$ candidate points
\[
 p_i=(x_i,1-x_i^2),\quad x_i=\frac{i-1}{69},\quad i=1,\ldots,70.
\]
These points form an ordered biobjective Pareto-front sample on the curve $f_2=1-f_1^2$.  Under the $\ell_1$ norm, Theorem~\ref{thm:l1-curve-uniform} applies through the induced line coordinate
\[
 t_i=x_i-(1-x_i^2)=x_i^2+x_i-1,
\]
so the discrete subset-selection problem is exactly the ordered line problem for the finite set $\{t_1,\ldots,t_{70}\}\subset[-1,1]$.

We set $\beta=1$ and select $k=10$ points.  Using the scaled gap-sum objective and Bellman recursion from~\cite{Emmerich2026PolynomialTime} gives $kn=700$ dynamic-programming states and
\[
 (k-1)\frac{n(n-1)}{2}=9\cdot\frac{70\cdot 69}{2}=21{,}735
\]
predecessor transitions.  It returns the optimal index set
\[
 I^\star=\{1,14,24,32,40,47,53,59,65,70\}.
\]
Thus the selected abscissae are
\[
 x_{I^\star}=(0,\ 0.18841,\ 0.33333,\ 0.44928,\ 0.56522,\ 0.66667,\ 0.753623,\ 0.84058,\ 0.92754,\ 1),
\]
and the corresponding induced line coordinates are
\[
 t_{I^\star}=(-1,\ -0.7760,\ -0.55556,\ -0.3488,\ -0.1153,  0.1111,\ 0.32157,\\ 0.54715,\ 0.78786,\ 1).
\]

Hence the consecutive selected $\ell_1$-gaps are
\[
 \Delta=(0.22390,\ 0.22054,\ 0.20668,\ 0.23356,\ 0.22642,\ 0.21046,\ 0.225583,\ 0.24071,\ 0.21214).
\]
Their mean is exactly
\[
 \frac{t_{70}-t_1}{k-1}=\frac{2}{9}\approx 0.222222,
\]
which is the uniform continuous spacing on $[-1,1]$, and the largest absolute deviation from this target is approximately $0.018484$.  The corresponding optimal discrete diversity value is
\[
 D_1(I^\star)=1+\sum_{r=1}^{9}\tanh\!\left(\frac{\Delta_r}{2}\right)\approx 1.995878.
\]
For comparison, the continuous optimum on the whole front is
\[
 1+9\tanh\!\left(\frac{1}{9}\right)\approx 1.995905,
\]
so the dense $70$-point candidate set comes very close to the ideal uniformly spaced solution.  The corresponding candidate set and selected subset are shown in Figure~\ref{fig:dense-front-example}.

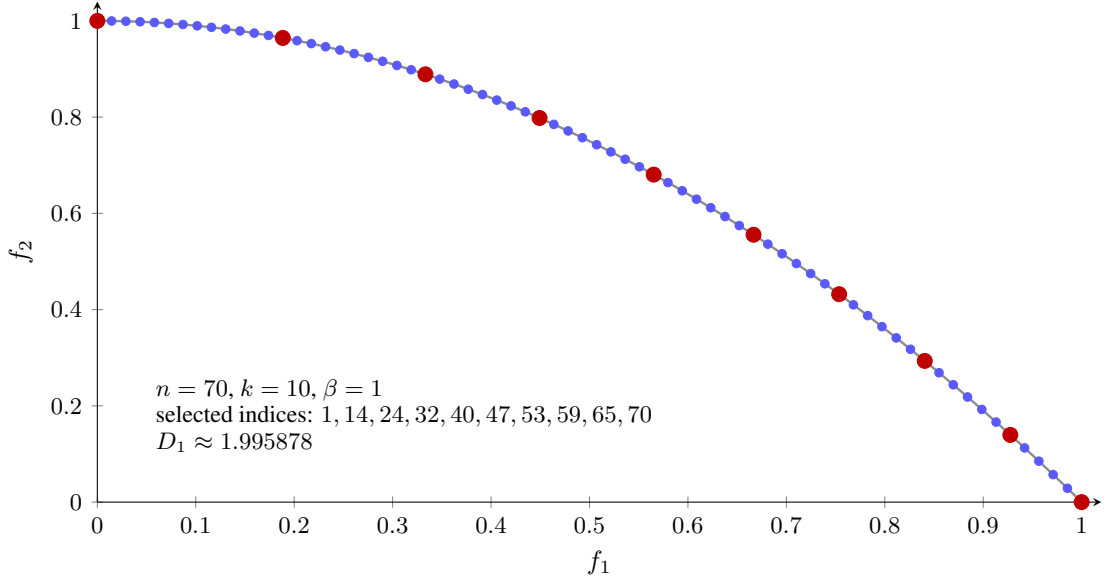
\begin{figure}[t]
\centering
\begin{tikzpicture}
\begin{axis}[
  width=0.9\textwidth,
  height=8.2cm,
  xmin=0, xmax=1.02,
  ymin=0, ymax=1.04,
  axis lines=left,
  xlabel={$f_1$},
  ylabel={$f_2$},
  enlargelimits=false,
  tick label style={font=\small},
  label style={font=\normalsize},
  clip=false
]

\addplot[domain=0:1,samples=201,smooth,black!45,line width=0.9pt] {1-x^2};

\addplot[
  only marks,
  mark=*,
  mark size=1.6pt,
  blue!65
] coordinates {
(0.000000,1.000000)
(0.014493,0.999790)
(0.028986,0.999160)
(0.043478,0.998110)
(0.057971,0.996639)
(0.072464,0.994749)
(0.086957,0.992439)
(0.101449,0.989708)
(0.115942,0.986557)
(0.130435,0.982987)
(0.144928,0.978996)
(0.159420,0.974585)
(0.173913,0.969754)
(0.188406,0.964503)
(0.202899,0.958832)
(0.217391,0.952741)
(0.231884,0.946230)
(0.246377,0.939298)
(0.260870,0.931947)
(0.275362,0.924176)
(0.289855,0.915984)
(0.304348,0.907372)
(0.318841,0.898341)
(0.333333,0.888889)
(0.347826,0.879017)
(0.362319,0.868725)
(0.376812,0.858013)
(0.391304,0.846881)
(0.405797,0.835329)
(0.420290,0.823356)
(0.434783,0.810964)
(0.449275,0.798152)
(0.463768,0.784919)
(0.478261,0.771267)
(0.492754,0.757194)
(0.507246,0.742701)
(0.521739,0.727788)
(0.536232,0.712455)
(0.550725,0.696702)
(0.565217,0.680529)
(0.579710,0.663936)
(0.594203,0.646923)
(0.608696,0.629490)
(0.623188,0.611636)
(0.637681,0.593363)
(0.652174,0.574669)
(0.666667,0.555556)
(0.681159,0.536022)
(0.695652,0.516068)
(0.710145,0.495694)
(0.724638,0.474900)
(0.739130,0.453686)
(0.753623,0.432052)
(0.768116,0.409998)
(0.782609,0.387524)
(0.797101,0.364629)
(0.811594,0.341315)
(0.826087,0.317580)
(0.840580,0.293426)
(0.855072,0.268851)
(0.869565,0.243856)
(0.884058,0.218442)
(0.898551,0.192607)
(0.913043,0.166352)
(0.927536,0.139677)
(0.942029,0.112581)
(0.956522,0.085066)
(0.971014,0.057131)
(0.985507,0.028775)
(1.000000,0.000000)
};

\addplot[
  only marks,
  mark=*,
  mark size=2.8pt,
  red!75!black
] coordinates {
(0.000000,1.000000)
(0.188406,0.964503)
(0.333333,0.888889)
(0.449275,0.798152)
(0.565217,0.680529)
(0.666667,0.555556)
(0.753623,0.432052)
(0.840580,0.293426)
(0.927536,0.139677)
(1.000000,0.000000)
};

\node[anchor=west,align=left,font=\small] at (axis cs:0.05,0.18)
{$n=70$, $k=10$, $\beta=1$\\selected indices: $1,14,24,32,40,47,53,59,65,70$\\$D_1\approx 1.995878$};

\end{axis}
\end{tikzpicture}
\caption{A dense ordered Pareto-front sample on $f_2=1-f_1^2$ with $70$ candidate points (blue) and the exact optimal $10$-point Solow--Polasky subset under the $\ell_1$ norm (red), computed by the polynomial-time Bellman recursion from~\cite{Emmerich2026PolynomialTime}.  The selected points are nearly uniformly spaced in the induced line coordinate $t=x-(1-x^2)=x^2+x-1$, with target spacing $2/9\approx 0.222222$ and realized gaps
$0.223903,0.220542,0.206679,0.233564,0.226423,0.210460,0.225583,0.240706,0.212140$.}
\label{fig:dense-front-example}
\end{figure}

\subsection{A disconnected ZDT3 finite-front example}

As a second illustration, we consider a disconnected ordered front.  The standard ZDT3 Pareto front is given by
\[
 f_2 = 1-\sqrt{f_1}-f_1\sin(10\pi f_1),
\]
restricted to its five Pareto-optimal components.  We generate $n=100$ candidate points by placing $20$ uniformly spaced candidates on each of the five $f_1$-intervals
\[
\begin{aligned}
 &[0,0.0830015349],\quad [0.1822287280,0.2577623634],\\
 &[0.4093136748,0.4538821041],\quad [0.6183967944,0.6525117038],\\
 &[0.8233317983,0.8518328654].
\end{aligned}
\]
The candidates are ordered by increasing $f_1$, with indices $1,\ldots,100$ assigned consecutively across the five components.

On this ordered front, the $\ell_1$ metric again induces the line coordinate
\[
 s=f_1-f_2.
\]
We set $\beta=1$ and select $k=20$ points maximizing Solow--Polasky diversity.  Applying the scaled Solow--Polasky gap-sum dynamic program from~\cite{Emmerich2026PolynomialTime} to solve the fixed-cardinality subset-selection problem on the ordered candidate set yields the optimal index set
\[
 I^\star=\{1,4,10,20,23,28,32,40,41,45,50,60,61,65,70,80,81,85,90,100\}.
\]
The resulting subset is shown in Figure~\ref{fig:zdt3-sp-subset}.  Its Solow--Polasky diversity value is
\[
 D_1(I^\star)\approx 2.310417.
\]
The target uniform spacing in the induced coordinate is approximately $0.138169$, and the largest absolute deviation of the realized consecutive gaps from this target is approximately $0.032652$.  The selected points approximate the uniform-spacing pattern while respecting the disconnected structure of the front.

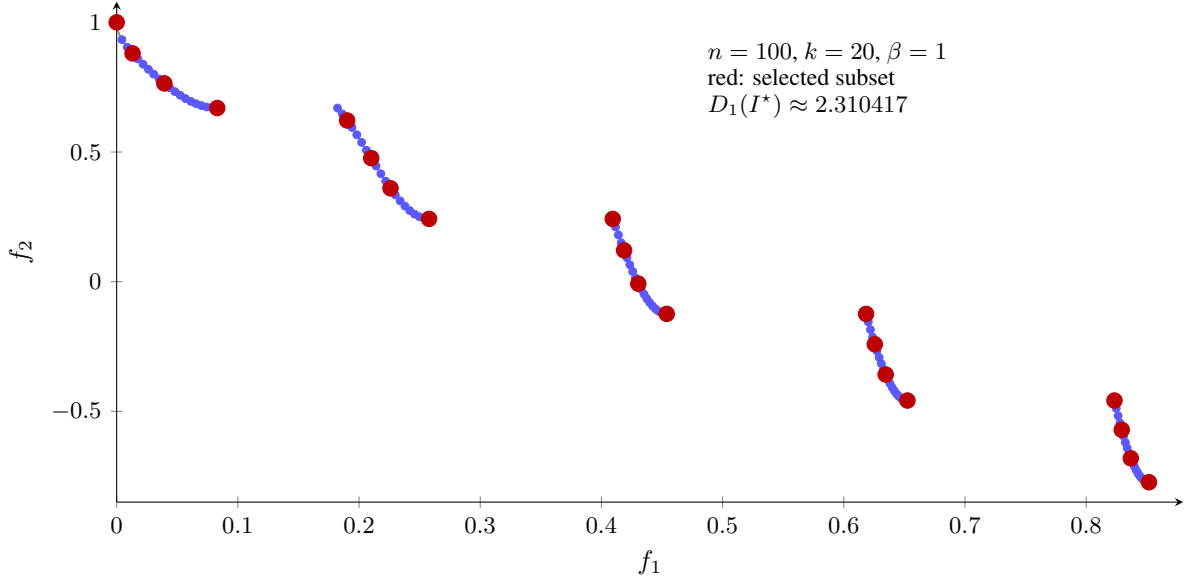
\begin{figure}[t]
\centering
\begin{tikzpicture}
\begin{axis}[
  width=0.95\textwidth,
  height=8.2cm,
  xmin=0, xmax=0.88,
  ymin=-0.85, ymax=1.08,
  axis lines=left,
  xlabel={$f_1$},
  ylabel={$f_2$},
  enlargelimits=false,
  tick label style={font=\small},
  label style={font=\normalsize},
  clip=false
]

\foreach \a/\b in {
  0.0000000000/0.0830015349,
  0.1822287280/0.2577623634,
  0.4093136748/0.4538821041,
  0.6183967944/0.6525117038,
  0.8233317983/0.8518328654
}{
  \addplot[
    domain=\a:\b,
    samples=120,
    smooth,
    black!45,
    line width=0.8pt
  ]
  {1 - sqrt(x) - x*sin(deg(10*pi*x))};

  \addplot[
    only marks,
    domain=\a:\b,
    samples=20,
    mark=*,
    mark size=1.45pt,
    blue!65
  ]
  {1 - sqrt(x) - x*sin(deg(10*pi*x))};
}

\addplot[
  only marks,
  mark=*,
  mark size=2.9pt,
  red!75!black
] coordinates {
(0.000000000,  1.000000000)
(0.013105506,  0.880276059)
(0.039316517,  0.764593318)
(0.083001535,  0.669652357)
(0.190179637,  0.621651164)
(0.210056909,  0.476412118)
(0.225958727,  0.360132712)
(0.257762363,  0.242161085)
(0.409313675,  0.242161085)
(0.418696502,  0.120902501)
(0.430425036, -0.007636135)
(0.453882104, -0.124218445)
(0.618396794, -0.124218445)
(0.625578881, -0.241257552)
(0.634556488, -0.357915810)
(0.652511704, -0.458263326)
(0.823331798, -0.458263326)
(0.829332023, -0.571243811)
(0.836832304, -0.681030545)
(0.851832865, -0.773369012)
};

\node[anchor=north west,align=left,font=\small] at (axis cs:0.48,0.96)
{$n=100$, $k=20$, $\beta=1$\\
red: selected subset\\
$D_1(I^\star)\approx 2.310417$};

\end{axis}
\end{tikzpicture}
\caption{Disconnected ZDT3 Pareto-front example with $100$ candidate points, shown in blue, and the exact optimal $20$-point Solow--Polasky subset under the $\ell_1$ norm, shown in red.  The candidates are ordered by increasing $f_1$, with $20$ uniformly spaced candidates on each of the five Pareto-front components.  For $\beta=1$, the selected index set is $I^\star=\{1,4,10,20,23,28,32,40,41,45,50,60,61,65,70,80,81,85,90,100\}$, and the corresponding diversity value is $D_1(I^\star)\approx 2.310417$.}
\label{fig:zdt3-sp-subset}
\end{figure}

\section{Reproducibility}
\label{sec:reproducibility}

A reproducibility package containing the computations used in the two examples is available at
\[
\href{https://github.com/emmerichmtm/SPUniformLinesAndOrderedPFs}{\texttt{github.com/emmerichmtm/SPUniformLinesAndOrderedPFs}}.
\]
\section{Conclusion}

The main message of this paper is that, for the exponential similarity kernel on ordered one-dimensional metric structures, Solow--Polasky diversity has an exact and transparent spacing principle.  On the line, the finite magnitude formula decomposes the diversity into a sum of gap terms, and the strict concavity of $t\mapsto \tanh(\beta t/2)$ makes equal gaps optimal by Jensen's inequality.  Conversely, exact adjacent-gap additivity of the baseline-corrected diversity forces the exponential family among normalized non-increasing kernels, with $\beta=-\log K(1)$.  Thus the continuous fixed-cardinality optimum on an interval is not merely well spread in an informal sense: it is characterized exactly by uniform consecutive distances.

This line result also explains the behavior on monotone Pareto fronts under the $\ell_1$ norm.  Such fronts are isometric to intervals through their accumulated coordinate-wise change, so the same proof applies without additional geometric assumptions.  The resulting recommendation is to interpret uniformity on an ordered Pareto front as uniformity in the induced $\ell_1$ line coordinate, not necessarily as uniformity in the original parameter or in Euclidean arclength.

For finite candidate sets, the exact continuous pattern may be unavailable.  Nevertheless, the isometry still reduces the ordered $\ell_1$ case to a finite ordered subset problem on the line, where dynamic programming can identify the feasible subset whose gap pattern best realizes the continuous uniform-spacing principle.  This separates the tractable ordered case from the more difficult general metric and Euclidean-plane subset-selection problems.

\appendix
\section{A finite Jensen inequality derivation}
\label{app:jensen}

This appendix derives the specific finite uniform-average form of Jensen's inequality used in the proof of Theorem~\ref{thm:unit-line-uniform}.  For the gap-vector problem considered here, the result says that, for a concave gap contribution function, replacing any feasible gap vector by its uniform average-gap vector gives an objective value no smaller than before.  The general inequality is a standard result in convex analysis and classical inequalities; see, for instance, Hardy, Littlewood, and P\'olya~\cite{HardyLittlewoodPolya1952}.

\begin{proposition}[Finite Jensen inequality for uniform averages for concave functions]
Let $I\subseteq\mathbb{R}$ be an interval and let $f:I\to\mathbb{R}$ be concave.  For any $n\geq 1$ and any points $u_1,\ldots,u_n\in I$,
\[
  \frac{1}{n}\sum_{i=1}^n f(u_i)
  \leq
  f\!\left(\frac{1}{n}\sum_{i=1}^n u_i\right).
\]
If $f$ is strictly concave and $n\geq2$, equality holds if and only if
\[
  u_1=u_2=\cdots=u_n.
\]
\end{proposition}

\begin{proof}
The defining property of a concave function is that its graph lies above every chord.  Equivalently, for any $u,v\in I$ and any $\lambda\in[0,1]$,
\begin{equation}
\label{eq:two-point-concavity}
  f\bigl(\lambda u+(1-\lambda)v\bigr)
  \geq
  \lambda f(u)+(1-\lambda)f(v).
\end{equation}
If $f$ is strictly concave, then the inequality is strict whenever $0<\lambda<1$ and $u\neq v$.

We prove the stated finite version by induction on $n$.  For $n=1$, the assertion is an equality.  For $n=2$, it is exactly \eqref{eq:two-point-concavity} with $\lambda=1/2$.

Assume the result has been proved for $n-1$ points, where $n\geq3$.  Define the average of the first $n-1$ points by
\[
  \bar u_{n-1}=\frac{1}{n-1}\sum_{i=1}^{n-1}u_i,
\]
and the average of all $n$ points by
\[
  \bar u_n=\frac{1}{n}\sum_{i=1}^n u_i.
\]
Then
\[
  \bar u_n=\frac{n-1}{n}\bar u_{n-1}+\frac{1}{n}u_n.
\]
Applying the two-point concavity inequality to the two inputs $\bar u_{n-1}$ and $u_n$ gives
\[
  f(\bar u_n)
  \geq
  \frac{n-1}{n}f(\bar u_{n-1})+\frac{1}{n}f(u_n).
\]
By the induction hypothesis applied to $u_1,\ldots,u_{n-1}$,
\[
  f(\bar u_{n-1})
  \geq
  \frac{1}{n-1}\sum_{i=1}^{n-1}f(u_i).
\]
Substituting this lower bound into the previous inequality yields
\[
  f(\bar u_n)
  \geq
  \frac{n-1}{n}\cdot\frac{1}{n-1}\sum_{i=1}^{n-1}f(u_i)
  +\frac{1}{n}f(u_n)
  =
  \frac{1}{n}\sum_{i=1}^n f(u_i).
\]
This proves Jensen's inequality for all finite uniform averages.

It remains to record the equality condition used in the main proof.  Suppose $f$ is strictly concave and equality holds for $n\geq2$.  In the induction step, equality can hold only if equality holds in the induction hypothesis and in the final two-point concavity step.  The induction hypothesis forces
\[
  u_1=\cdots=u_{n-1},
\]
and strict concavity in the two-point step forces
\[
  \bar u_{n-1}=u_n.
\]
Together these conditions give $u_1=\cdots=u_n$.  Conversely, if all $u_i$ are equal, then both sides of Jensen's inequality are obviously equal.  This completes the proof.
\end{proof}

\section{Detailed proof for the converse characterization}
\label{app:converse-details}

This appendix expands the proof of Proposition~\ref{prop:converse-exponential-kernel}.  The purpose is to show explicitly how three-point adjacent-gap additivity leads first to the multiplicative identity
\[
  K(a+b)=K(a)K(b),
\]
and then, through the continuous Cauchy equation, to the exponential form.

\subsection*{Step 1: the two-point excess diversity}

For two points at distance $t$, write $r=K(t)$.  The similarity matrix is
\[
  Z_2=
  \begin{pmatrix}
  1&r\\ r&1
  \end{pmatrix},
  \qquad
  Z_2^{-1}=\frac{1}{1-r^2}
  \begin{pmatrix}
  1&-r\\ -r&1
  \end{pmatrix}.
\]
Therefore
\[
\begin{aligned}
  D_K(\{0,t\})
  &=\mathbf{1}^\top Z_2^{-1}\mathbf{1} \\
  &=\frac{1}{1-r^2}(1,1)
    \begin{pmatrix}
    1&-r\\ -r&1
    \end{pmatrix}
    \begin{pmatrix}1\\1\end{pmatrix} \\
  &=\frac{2-2r}{1-r^2}
   =\frac{2}{1+r}.
\end{aligned}
\]
Thus the baseline-corrected excess diversity is
\[
  E_K(\{0,t\})
  =D_K(\{0,t\})-1
  =\frac{2}{1+r}-1
  =\frac{1-r}{1+r}.
\]

\subsection*{Step 2: the three-point excess diversity}

Fix $a,b>0$ and write
\[
  u=K(a),\qquad v=K(b),\qquad w=K(a+b).
\]
For $\{0,a,a+b\}$, the similarity matrix is
\[
  Z_3=
  \begin{pmatrix}
  1&u&w\\
  u&1&v\\
  w&v&1
  \end{pmatrix}.
\]
Set
\[
  \Delta:=\det Z_3=1+2uvw-u^2-v^2-w^2.
\]
The cofactor formula gives
\[
  Z_3^{-1}=\frac{1}{\Delta}
  \begin{pmatrix}
  1-v^2 & vw-u & uv-w\\
  vw-u & 1-w^2 & uw-v\\
  uv-w & uw-v & 1-u^2
  \end{pmatrix}.
\]
Since $D_K(\{0,a,a+b\})=\mathbf{1}^\top Z_3^{-1}\mathbf{1}$, summing all entries of this inverse yields
\[
\begin{aligned}
  D_K(\{0,a,a+b\})
  =\frac{1}{\Delta}\bigl(&3-u^2-v^2-w^2 \\
  &{}-2u-2v-2w+2uv+2uw+2vw\bigr).
\end{aligned}
\]
Subtracting the singleton baseline gives
\[
\begin{aligned}
  E_K(\{0,a,a+b\})
  &=D_K(\{0,a,a+b\})-1 \\
  &=\frac{2(1-u)(1-v)(1-w)}{1+2uvw-u^2-v^2-w^2}.
\end{aligned}
\]
The assumed adjacent-gap additivity of the excess diversity is
\[
  E_K(\{0,a,a+b\})
  =E_K(\{0,a\})+E_K(\{0,b\}).
\]
Using the two-point formula from Step 1, this becomes
\[
  \frac{2(1-u)(1-v)(1-w)}{1+2uvw-u^2-v^2-w^2}
  =\frac{1-u}{1+u}+\frac{1-v}{1+v}.
\]
The right-hand side simplifies to
\[
  \frac{1-u}{1+u}+\frac{1-v}{1+v}
  =\frac{(1-u)(1+v)+(1-v)(1+u)}{(1+u)(1+v)}
  =\frac{2(1-uv)}{(1+u)(1+v)}.
\]
Hence the additivity identity is equivalent to
\[
  \frac{2(1-u)(1-v)(1-w)}{1+2uvw-u^2-v^2-w^2}
  =\frac{2(1-uv)}{(1+u)(1+v)}.
\]
Cancelling the common factor $2$ and cross-multiplying gives
\[
  (1-u)(1-v)(1-w)(1+u)(1+v)
  =(1-uv)(1+2uvw-u^2-v^2-w^2).
\]
Expanding and collecting terms yields the factorization
\[
  (uv-w)\bigl(u^2-uvw-uv+v^2+w-1\bigr)=0.
\]
Consequently, either
\[
  w=uv,
\]
or
\[
  u^2-uvw-uv+v^2+w-1=0.
\]
Solving the second equation for $w$ gives
\[
  w=\frac{1+uv-u^2-v^2}{1-uv}.
\]
The second branch is incompatible with the non-increasing character of $K$.  Indeed, after interchanging $a$ and $b$ if necessary, assume $u\geq v$.  Then $0<v\leq u<1$ and, since $a+b>b$, non-increase gives $w\leq v$.  But on the second branch,
\[
\begin{aligned}
  w-v
  &=\frac{1+uv-u^2-v^2}{1-uv}-v \\
  &=\frac{1+uv-u^2-v^2-v+uv^2}{1-uv} \\
  &=\frac{(1-u)(u+1-v-v^2)}{1-uv}.
\end{aligned}
\]
Here $1-u>0$ and $1-uv>0$.  Moreover, $u\geq v$ implies
\[
  u+1-v-v^2\geq v+1-v-v^2=1-v^2>0.
\]
Thus $w-v>0$, contradicting $w\leq v$.  Therefore the only possible branch is
\[
  w=uv,
\]
which is precisely
\[
  K(a+b)=K(a)K(b).
\]

\subsection*{Step 3: the continuous Cauchy equation}

It remains to identify all continuous positive kernels satisfying
\[
  K(a+b)=K(a)K(b).
\]
Define
\[
  f(t)=-\log K(t).
\]
Because $K(t)>0$ and $K$ is continuous, the function $f$ is continuous.  Moreover,
\[
\begin{aligned}
  f(a+b)
  &=-\log K(a+b) \\
  &=-\log\bigl(K(a)K(b)\bigr) \\
  &=-\log K(a)-\log K(b) \\
  &=f(a)+f(b).
\end{aligned}
\]
Thus $f$ satisfies Cauchy's additive equation on $[0,\infty)$:
\[
  f(a+b)=f(a)+f(b).
\]
We recall the elementary continuous case, i.e., Cauchy's additive equation.  First,
\[
  f(0)=f(0+0)=f(0)+f(0),
\]
so $f(0)=0$.  For every integer $n\geq1$,
\[
  f(nt)=f(t+\cdots+t)=n f(t).
\]
In particular,
\[
  f(n)=n f(1).
\]
Now let $m,n$ be positive integers.  Since
\[
  n\cdot\frac{m}{n}=m,
\]
we get
\[
  n f\!\left(\frac{m}{n}\right)
  =f(m)
  =m f(1).
\]
Therefore
\[
  f\!\left(\frac{m}{n}\right)
  =\frac{m}{n}f(1).
\]
So $f(q)=qf(1)$ for every nonnegative rational number $q$.  Finally, let $t\geq0$ be real and choose rational numbers $q_j\geq0$ with $q_j\to t$.  By continuity,
\[
  f(t)=\lim_{j\to\infty} f(q_j)
      =\lim_{j\to\infty} q_j f(1)
      =t f(1).
\]
Thus every continuous solution is linear:
\[
  f(t)=t f(1).
\]
Returning to $K$, this gives
\[
  -\log K(t)=t(-\log K(1)).
\]
With
\[
  \beta=-\log K(1),
\]
we obtain
\[
  K(t)=e^{-\beta t}.
\]
Since the assumptions include $K(1)<1$, we have $\beta>0$.  The same formula also shows uniqueness of the parameter for the given kernel: once $K(1)$ is fixed, the value of $\beta$ is fixed by $\beta=-\log K(1)$.

\end{document}